\documentclass[12pt]{article}

\usepackage{amsmath} 
\usepackage{amssymb}
\usepackage{amsthm}
\usepackage{graphicx}
\usepackage{amscd}
\usepackage{epic, eepic}
\usepackage{url}
\usepackage{xcolor}
\usepackage{todonotes}
\usepackage{enumerate}
\usepackage{enumitem}
\usepackage{comment}
\usepackage{hyperref}
\usepackage[utf8]{inputenc}
\usepackage{tabularray}

\newtheorem{theorem}{\bf Theorem}[section]
\newtheorem{lemma}[theorem]{\bf Lemma}

\newtheorem{prop}[theorem]{\bf Proposition}

\theoremstyle{definition}

\newcommand{\Z}{\mathbb{Z}}

\newcommand{\dd}{\frac{p-1}{2}}

\newcommand{\F}{\mathbb{F}}

          \usepackage{array}

\usepackage{todonotes}

\title{New bound on small range sum polynomials of degree $\frac{p-1}2$}

\author{ Ádám Markó \thanks{E\"otv\"os Lor\'and University, Institute of Mathematics, Budapest, Hungary
E-mail: {\tt marqadam@gmail.com}} 
\thanks{The research was carried out at the Erdős Center in the framework of the FOURIER ANALYSIS AND ADDITIVE PROBLEMS semester}.}

\date{}

\begin{document}
\maketitle
\begin{abstract}
The polynomials of degree $\frac{p-1}{2}$ of range sum $p$ was determined  in \cite{negyen} for large enough primes.  
We extend this result by reducing the lower bound for the primes to $23$ by introducing a new and elementary way of estimating sums of Legendre symbols.  
\end{abstract}

\section{Introduction}
The main question  investigated in this paper is to derive connection between the range of a function determined by a polynomial over $\mathbb{F}_p$, where $p$ is a prime, and the degree of the polynomial itself. We give lower estimates for the degree of a polynomial whose range sum is $p$. This allows us to give a new proof of certain direction problems.

Let  $S$ be a subset of $AG(2,p)$. For two different elements $s_1,s_2 \in S$, the difference $s_1-s_2$ determines a point in the projective line $PG(1,p)$. In this case, the corresponding point of the projective line is a \textit{direction} determined by $S$. We are interested in the number of directions determined by $S$.
An easy pigeonhole argument shows that sets of cardinality larger than $p$ determine every direction  so this question remains interesting for sets of relatively small cardinality. 
One of the earliest use of polynomial method to handle combinatorial problems were introduced by Rédei's \cite{redei} (whose result was extended by Megyesi) to prove that a set of size $p$ in the finite affine plane $\mathbb{F}_p^2$ is either a line or determines at least $\frac{p+3}{2}$ directions. 
The original proofs relies on the usage of Rédei's polynomial and heavily builds on the theory lacunary polynomials. 
Rédei's result was also independently obtained by Dress, Klin and Muzychuk \cite{dressklinmuzychuk} on a way of providing a new proof for an old theorem of Burnside's describing transitive permutation groups of degree $p$. 
A Fourier transformation based proof was given by Lev \cite{lev}.

A new proof of Somlai \cite{som} uses  Rédei's polynomials and rely on the new notion of \textit{projection polynomials}, introduced in \cite{KS}, which can be considered as an intermediate step towards calculating the Fourier transform of the characteristic function of a set $S$. The main new ingredient of the new approach is the fact that non-constant polynomials having small range sum must have very large degree, at least $\frac{p-1}{2}$. It was conjectured in \cite{som} that the polynomials of range sum $p$ of degree $\frac{p-1}{2}$ is affine equivalent to the polynomial $x^{\frac{p-1}{2}}+1$. This turned out to be false since $\frac{p+1}{2}(x^{\frac{p-1}{2}} + 1)$ also satisfies the requirements. It remained plausible to believe that these are the only polynomial with basically minimal range sum of smallest possible degree if we exclude constant polynomials. 

It was proved in \cite{negyen} that the conjecture holds for primes larger than $7.5*10^6$. 
The proof uses Weil bounds in order to estimate certain sums of Legendre symbols. The present paper introduces a new way of estimating similar exponential sums and avoids the usage of heavyweight results, replacing Weil bounds by a Cauchy Schwartz estimate and a better understanding of 'small errors'. Furthermore, the proof is not only elementary but more efficient so we obtain a much better bound for which the uniqueness of the polynomials is proved. The main result of the paper is as follows.

Let $f$ be a polynomial in $\mathbb{F}_p[x]$, where $\mathbb{F}_p$ denotes the field of size $p$, where $p$ is a prime. Identify the elements of $\mathbb{F}_p$ with the set of integers $\{0,1, \ldots, p-1 \}$. This allows us to formulate the following theorem. 
\begin{theorem}\label{thm1}
Let $p>23$ be a prime. Assume $f \in \mathbb{F}_p[x]$ is a polynomial, which defines a function  from $\mathbb{F}_p$ to $\{0,1,\ldots ,p-1 \}$. Assume that $\sum_{x\in \F_p} f(x)=p$. Then $deg(f)\ge \dd$. 
\end{theorem}

Sets determining exactly $\frac{p+3}{2}$ directions exist and they were explicitly described by Lovász and Schrijver \cite{lovaszschrijver}. They proved that up to an affine transformations there is a unique set of this sort. It was proved in \cite{negyen} that if the 'uniqueness' for the polynomials of degree $\frac{p-1}{2}$ holds as in Theorem \ref{thm1}, then an easy Fourier transformation argument gives the uniqueness result for sets in $AG(2,p)$ determining exactly $\frac{p+3}{2}$ directions,  originally proved by Lovász and Schrijver \cite{lovaszschrijver}. Thus we obtain a new proof for this uniqueness result of Lovász and Schrijver for primes larger than 23. 

\section{Notation and earlier lemmas}
Let $S$ be a subset of $\mathbb{F}_p^2$, where $p$ is a prime and $\mathbb{F}_p$ denotes the field of $p$ elements. We describe the set of directions determined by $S$ in the following way. 
Let us consider the nonzero elements of $S-S$. For each nonzero vector in $\mathbb{F}_p^2$ we can assign an element of the projective line $PG(1,p)$ by considering two nonzero vectors equivalent if they are nonzero multiples of each other. 

We will treat the elements of $\mathbb{F}_p$ in two different ways. In some cases we identify them with the set $\{0,1, \ldots , p-1 \}$, which is a subset of the integers. We exploit this identification to talk about the range sum of a polynomial (function). Let $f$ be a polynomial in $\F_p[x]$. Every element $f(x)$ of the range can be considered as an element of $\{0,1, \ldots, p-1 \} \subset \Z$, so we may sum the elements of the range of $f$ as integers. We will consider those polynomials where the sum of the range is equal to $p$ so we write $\sum_{x \in \F_p} f(x)=_{\Z} p$, indicating that the numbers we sum are elements of $\Z$. 

The Legendre symbol is denoted by $(\frac{a}{p})$. It is equal to 1 if and only of $a$ is a quadratic residue  
modulo $p$
and it is $-1$ if $a$ is a quadratic nonresidue, and $(\frac{0}{p})=0$.

We will rely on the results of \cite{negyen} so we first recall the essential lemmas that are needed to start the new investigation. 
\begin{lemma}
Let $f$ be a polynomial of degree $\frac{p-1}{2}$ of range sum $p$. Then $f$ is completely reducible. 
\end{lemma}
Let us denote the set of roots of $f$ by $\alpha_1, \ldots , \alpha_{\dd} $. Let us define a multiset $B$, which contains those elements $\beta$ such that $f(\beta)>1$. The multiplicity of $\beta \in B$  is $f(\beta)-1$ For more precise definition, see \cite{negyen}. 
\begin{lemma}
Let $f$ be a polynomial of range sum $p$ of degree $\frac{p-1}{2}$.
For any $\gamma \in \mathbb{F}_p$ we have
   $$ \sum_{i=1}^{\frac{\dd}{2}}\bigg( \frac{\alpha_i-\gamma}{p} \bigg) =   \sum_{i=1}^{\frac{\dd}{2}}\bigg(\frac{\alpha_i-\gamma}{p}\bigg) + r_{\gamma},  $$
   where $r_{\gamma}$ is either equal to the leading coefficient $c$ of $f$, or it is equal to $c-p$, where $c$ is also handled as an integer in $\{1,2,\ldots, p-1 \}$.
\end{lemma}

\section{New estimate}

\begin{theorem}\label{thme1}
    For any $A \subset \mathbb F_p $ the following inequality holds  
\begin{displaymath}
     \sum_{\gamma \in \mathbb F_p} \bigg | \sum_{\alpha \in A } \bigg( \frac{\alpha - \gamma}p \bigg)  \bigg | \leq  \sqrt{p} \sqrt{|A|} \sqrt{p-|A|} \leq \frac{1}{2} p^\frac{3}{2}.
\end{displaymath}
\end{theorem}

\begin{proof}
  Let $A \subset \mathbb F_p $. Cauchy-Schwarz inequality gives 
\begin{displaymath}
  \sum_{\gamma \in \mathbb F_p} \bigg | \sum_{\alpha \in A } \bigg( \frac{\alpha - \gamma}p \bigg)  \bigg | \leq \sqrt{p} \sqrt{\sum_{\gamma \in \F_p} \bigg( \sum_{\alpha \in A } \bigg( \frac{\alpha - \gamma}p \bigg) \bigg)^2}.
\end{displaymath}
Now we estimate this term.
\begin{equation}\label{eq:3}
    \sum_{\gamma \in \F_p} \bigg( \sum_{\alpha \in A} \bigg( \frac{\alpha - \gamma}p \bigg) \bigg)^2 =  \sum_{\gamma \in \F_p} \bigg( \sum_{\alpha \in A} \bigg( \frac{\alpha - \gamma}p \bigg)^2 + 2 \sum_{\alpha_1, \alpha_2 \in A  ,\alpha_1 \neq \alpha_2} \bigg( \frac{\alpha_1 - \gamma}p \bigg) \bigg( \frac{\alpha_2 - \gamma}p \bigg) \bigg) =
\end{equation}
\begin{displaymath}
    = (p-1) \lvert A \lvert  
    + 2 \sum_{\alpha_1 \neq \alpha_2} \sum_{\gamma \in \F_p} \bigg( \frac{\alpha_1 - \gamma}p \bigg) \bigg( \frac{\alpha_2 - \gamma}p  \bigg) \approx p \lvert A \lvert.
\end{displaymath}
Now we claim that for every $\alpha_1 \ne \alpha_2 \in \mathbb{F}_p$ 
\begin{displaymath}
    \sum_{\gamma \in \mathbb F_p} \bigg( \frac{\alpha_1 - \gamma}p \bigg) \bigg( \frac{\alpha_2 - \gamma}p \bigg) =-1.
\end{displaymath}
It is important to note that the previous expression is negative and this is what we only use.

Let $$A_{e_1, e_2}=  \bigg \{ \gamma \in \mathbb F_p ~\big| ~ \bigg( \frac {\alpha_1 - \gamma} p \bigg) = e_1, \bigg(\frac{\alpha_2-\gamma}p\bigg) = e_2   \bigg \},$$
where $e_1, e_2\in \{\pm 1\}.$

The following properties can be derived from elementary knowledge on quadratic residues. In particular, some of these numbers coincide with the parameters of Paley graphs.
\SetTblrInner{rowsep=3pt}
\begin{table}[h]
    \centering

    \[
\small
\begin{matrix}
\begin{array}{|c|c|c|}
    \hline
    \left(\frac{\alpha_1 - \alpha_2}{p} \right) = 1 & p \equiv 1 \pmod{4} & p \equiv 3 \pmod{4} \\
    \hline
    |A_{1, 1}| & \dfrac{p-5}{4} & \dfrac{p-3}{4} \\
    \hline
    |A_{1, -1}| & \dfrac{p-1}{4} & \dfrac{p+1}{4} \\
    \hline
    |A_{-1, 1}| & \dfrac{p-1}{4} & \dfrac{p-3}{4} \\
    \hline
    |A_{-1, -1}| & \dfrac{p-1}{4} & \dfrac{p-3}{4} \\
    \hline
\end{array}
&
\begin{array}{|c|c|c|}
    \hline
    \left(\frac{\alpha_1 - \alpha_2}{p} \right) = -1 & p \equiv 1 \pmod{4} & p \equiv 3 \pmod{4} \\
    \hline
    |A_{1, 1}| & \dfrac{p-1}{4} & \dfrac{p-3}{4} \\
    \hline
    |A_{1, -1}| & \dfrac{p-1}{4} & \dfrac{p-3}{4} \\
    \hline
    |A_{-1, 1}| & \dfrac{p-1}{4} & \dfrac{p+1}{4} \\
    \hline
    |A_{-1, -1}| & \dfrac{p-5}{4} & \dfrac{p-3}{4} \\
    \hline
\end{array}
\end{matrix}
\]
    \caption{Intersection size of translates of quadratic (non)residues.}
    \label{tqr} 
\end{table}

By this table
\begin{displaymath}
    \sum_{\gamma \in \mathbb F_p} \bigg( \frac{\alpha_1 - \gamma}p \bigg) \bigg( \frac{\alpha_2 - \gamma}p \bigg) = |A_{1,1}| - |A_{1,-1}|-|A_{-1,1}|+ |A_{-1,-1}| =  -1,
\end{displaymath}
as it claimed earlier. 

By equation \eqref{eq:3} 
\begin{displaymath}
    \sum_{\gamma \in \mathbb F_p} \bigg | \sum_{\alpha \in A } \bigg( \frac{\alpha - \gamma}p \bigg)  \bigg| \leq \sqrt{p} \sqrt{\left( (p-1)|A| - 2\binom{|A|}{2} \right)  } \leq \sqrt{p} \sqrt{|A|} \sqrt{p-|A|} . 
\end{displaymath}
This expression is maximal if $|A|=\frac{p}{2}$ so we obtain $\sqrt{p} \sqrt{|A|} \sqrt{p-|A|}  \leq \frac{1}{2}p^{\frac32 }$ .
\end{proof}
    
\begin{prop}\label{prope2}
    For any $A, \Gamma \subset \mathbb F_p $
\begin{displaymath}
     \bigg | \sum_{\gamma \in  \Gamma} \sum_{\alpha \in A } \bigg( \frac{\alpha - \gamma}p \bigg) \bigg| \leq \frac{1}{2} \sqrt{p} \sqrt{|A|} \sqrt{p-|A|}  \leq \frac{1}{4} p^\frac{3}{2}.
\end{displaymath}
\end{prop}

\begin{proof}
    It is easy to see that  
\begin{displaymath}
          \sum_{ \alpha \in A } \sum_{\gamma \in \mathbb F _p  } \bigg( \frac{\alpha - \gamma}p \bigg)=
     \sum_{\gamma \in \mathbb F _p} \sum_{\alpha \in A } \bigg( \frac{\alpha - \gamma}p \bigg)  = 0.
\end{displaymath}
Hence
\begin{displaymath}
  \bigg|  \sum_{\gamma \in  \Gamma} \sum_{\alpha \in A } \bigg( \frac{\alpha - \gamma}p \bigg) \bigg | =  \bigg|  \sum_{\gamma \in \mathbb \mathbb \mathbb{F}_p \setminus \Gamma} \sum_{\alpha \in A } \bigg( \frac{\alpha - \gamma}p \bigg) \bigg | .
\end{displaymath}
It follows that 
\begin{displaymath}
     \bigg|  \sum_{\gamma \in \Gamma} \sum_{\alpha \in A } \bigg( \frac{\alpha - \gamma}p \bigg) \bigg | = \frac{1}{2}\bigg( \bigg|  \sum_{\gamma \in \Gamma} \sum_{\alpha \in A } \bigg( \frac{\alpha - \gamma}p \bigg) \bigg | + \bigg|  \sum_{\gamma \in \mathbb{F}_p \setminus \Gamma} \sum_{\alpha \in A } \bigg( \frac{\alpha - \gamma}p \bigg) \bigg | \bigg) \leq
\end{displaymath}
by the triangle inequality
\begin{equation*}
\begin{split}
     \frac{1}2 \bigg( \sum_{\gamma \in \Gamma} \bigg|  \sum_{\alpha \in A } \bigg( \frac{\alpha - \gamma}p \bigg) \bigg | + \sum_{\gamma \in  \mathbb{F}_p \setminus \Gamma} \bigg| \sum_{\alpha \in A } \bigg( \frac{\alpha - \gamma}p \bigg) \bigg | \bigg) = \frac{1}{2}  \sum_{\gamma \in \mathbb{F}_p } \bigg| \sum_{\alpha \in A } \bigg( \frac{\alpha - \gamma}p \bigg) \bigg |.
     \end{split}
\end{equation*}
By Theorem \ref{thme1} we estimate the last expression from above by $\frac{1}{2} \sqrt{p} \sqrt{|A|} \sqrt{p-|A|}  \leq \frac{1}{4} p^\frac{3}{2}$.
\end{proof}

Let $B$ be the multiset of the values where $f(x) > 1$, and let us decompose $B$ into homogeneous multisets
\begin{displaymath}
    B = \bigcup_{j= 1}^{n} B_j,
\end{displaymath}
where $B_j:=\{b_j,...,b_j\} $. Notice that $n$ denotes the number of different element of $B$.
Let $k_j := | B_j |$. 

The proof of the following Proposition is basically identical to the one of Theorem \ref{thme1}.
\begin{prop}
    For the $B$ multiset it holds that:
\begin{displaymath}
     \sum_{\gamma \in \mathbb F_p} \bigg | \sum_{\beta\in B } \bigg( \frac{\beta - \gamma}p \bigg)  \bigg | \leq p \sqrt{\sum_{j =1}^{n} k_j^2}.
\end{displaymath}
\end{prop}

\begin{proof}
Using again Cauchy-Schwarz inequality we obtain that
\begin{displaymath}
     \sum_{\gamma \in \mathbb F_p} \bigg | \sum_{\beta\in B } \bigg( \frac{\beta - \gamma}p \bigg)  \bigg | \leq \sqrt{p}\sqrt{\sum_{\gamma \in \mathbb F_p} \left( \sum_{\beta\in B } \left( \frac{\beta - \gamma}p \right ) \right)^2 }.
\end{displaymath}
Now we estimate this last term as follows. 
\begin{displaymath}
     \sum_{\gamma \in \mathbb F_p} \bigg ( \sum_{\beta\in B } \bigg( \frac{\beta - \gamma}p \bigg) \bigg)^2 = \sum_{\gamma \in \mathbb{F}_p} \sum_{j = 1}^{n} \bigg( \sum_{\beta \in B_j} \bigg( \frac{\beta-\gamma}{p} \bigg) \bigg)^2  +  2 \sum_{\beta_1 \neq \beta_2} \sum_{\gamma \in \mathbb{F}_p} \bigg( \frac{\beta_1 - \gamma}p \bigg) \bigg( \frac{\beta_2 - \gamma}p \bigg) .
\end{displaymath}
We may use again that the second term is negative to obtain the following upper bound. 
\begin{displaymath}
    p \sum_{j=1}^{n} k^2_j.
\end{displaymath}
It follows from the previous calculation that 
\begin{displaymath}
\sum_{\gamma \in \mathbb F_p} \bigg | \sum_{\beta\in B } \bigg( \frac{\beta - \gamma}p \bigg)  \bigg | \leq p \sqrt{\sum_{j=1}^{n} k^2_j}.
\end{displaymath}
\end{proof}
\subsection{Polynomials of degree $\frac{p-1}{2}$}
Let $f$ be a polynomial of degree $\frac{p-1}{2}$ and let $c$ denote the leading coefficient of $f$. 
The main aim of this section is to prove Theorem \ref{thm1}. In order to do so we prove the following. 
\begin{prop}\label{prop:LC}
The leading coefficient of $f$ can only be $ 1,\frac{p-1}{2}, \frac{p+1}{2}, \mbox{ or } p-1$.
    
\end{prop}
\begin{proof}
It is straighforward to see that if the leading coefficient of a polynomial $f(x)$ of degree $\dd$ is $c$, then the one of $f(ax)$  is $-c$ if $a$ is a quadratic nonresidue. Thus we assume that $1 \le c \le \dd$.

It was proved in subsection 3.2 in \cite{negyen} that
\begin{displaymath}
    \sum_{\alpha \in A} \bigg( \frac{\alpha-\gamma}p \bigg) \equiv \sum_{\beta \in B} \bigg(\frac{\beta-\gamma}{p} \bigg) + c \pmod{p}.
\end{displaymath}
Both sides of the previous equation can be considered as 
integers so these are the sum of $\pm 1, 0$'s. Thus we obtain that 

$$ \sum_{\alpha \in A} \bigg( \frac{\alpha-\gamma}p \bigg) = \sum_{\beta \in B} \bigg(\frac{\beta-\gamma}{p} \bigg) +r,$$
where $r=c$ or $r=c-p$
It was proved that in \cite{negyen} that $r=c$ occurs exactly $p-c$ times and $r=c-p$ occurs $c$ times. 

 If $ 1 < c \leq \frac{p}{4}$, then by Lemma 4.1 in \cite{negyen} there is at most one $\gamma$, such that $\sum \left( \frac{\alpha -\gamma}p \right) \leq - \frac{p-1}4 $. Since $c > 1$ there are at least 2 different $\gamma$ values, such that:
\begin{displaymath}
    \sum_{\alpha \in A} \bigg( \frac{\alpha-\gamma}p \bigg) = \sum_{\beta \in B} \bigg(\frac{\beta-\gamma}{p} \bigg) + c-p.
\end{displaymath}
Since $c\le \frac{p-1}{4}$ we have that $c-p <- \frac{3p+1}{4}$. On the other hand in at least one of these $c$ cases  $ \sum_{\alpha \in A} \left( \frac{\alpha-\gamma}p \right) - \sum_{\beta \in B}  \left(\frac{\beta-\gamma}{p} \right) \ge -\frac{p-1}{4}-\frac{p-1}{2} = - \frac{3p+1}4$, which is a contradiction.

From now on we assume $c > \frac{p}{4}$.

Let $\Gamma^+ \subset \mathbb{F}_p$ be the set of $\gamma$ values such that:
\begin{displaymath}
    \sum_{\alpha \in A} \bigg( \frac{\alpha-\gamma}p \bigg) = \sum_{\beta \in B} \bigg(\frac{\beta-\gamma}{p} \bigg) + c-p.
\end{displaymath}
We have that $|\Gamma^+| = c$.
By adding up the equations above.
\begin{equation}\label{eq:c(c-p)}
        \sum_{\alpha \in A} \sum_{\gamma \in \Gamma^+} \bigg( \frac{ \alpha-\gamma}p\bigg) =  \sum_{\gamma \in \Gamma^+} \sum_{\alpha \in A} \bigg( \frac{ \alpha-\gamma}p\bigg) =\sum_{\gamma \in \Gamma^+} \sum_{\beta \in B}\bigg( \frac{\beta- \gamma}p \bigg) + c(c-p).
\end{equation}

Since  $\frac{p-1}{3} \ge c > \frac{p}{4}$, it holds that $c(p-c) > \frac{3p^2}{16}$. By Proposition \ref{prope2}
\begin{displaymath}
\sum_{\gamma \in \Gamma^+ }\sum_{\alpha \in A} \bigg( \frac{\alpha-\gamma}{p}\bigg)  \geq - \frac{\sqrt{p |A|(p-|A|)}}2 = -\frac{\sqrt{p(p^2-1)}}4.
\end{displaymath}
Rearranging equation \eqref{eq:c(c-p)} gives
\begin{displaymath}
  \sum_{\gamma \in \Gamma^+} \sum_{\beta \in B }\bigg( \frac{\beta- \gamma}p \bigg) = c(p-c) +  \sum_{\gamma \in \Gamma^+ }\sum_{\alpha \in A} \bigg( \frac{\alpha-\gamma}{p}\bigg) > 
\end{displaymath}
\begin{displaymath}
 \frac{3p^2}{16} - \frac{\sqrt{p(p^2-1)}}4.
\end{displaymath}
For $p \geq 23$
\begin{equation}\label{eq:p^2/8}
    \sum_{\gamma \in \Gamma^+} \sum_{\beta \in B}\bigg( \frac{\beta- \gamma}p \bigg) > \frac{p^2}8.
\end{equation}

By swapping the sums we obtain
\begin{equation}\label{eq:p^28}
        \sum_{\gamma \in \Gamma^+} \sum_{\beta \in B}\bigg( \frac{\beta- \gamma}p \bigg) = \sum_{\beta \in B} \sum_{\gamma \in \Gamma^+} \bigg( \frac{\beta- \gamma}p \bigg).
\end{equation}

Let $\beta'$ be that, for which $\sum_{\gamma \in \Gamma^+} \left( \frac{\beta'- \gamma}p \right)$ is maximal, and let $t$ be such that:
\begin{displaymath}
     \sum_{\gamma \in \Gamma^+} \bigg( \frac{\beta'- \gamma}p \bigg) = \frac{p-1}4 + t.
\end{displaymath}
If $t$ is negative, then 
\begin{displaymath}
   \sum_{\gamma \in \Gamma^+}  \sum_{\beta \in B}\bigg( \frac{\beta- \gamma}p \bigg) < |\Gamma^+| \frac{p-1}{4} \le \frac{(p-1)^2}{8}< \frac{p^2}8,
\end{displaymath}
contradicting equation \eqref{eq:p^2/8}. Let us suppose that $t$ is positive. By Lemma 4.1 in \cite{negyen} we have 
\begin{displaymath}
     \sum_{\gamma \in \Gamma^+} \bigg( \frac{\beta"- \gamma}p \bigg) \leq \frac{p-1}4 - t+1,
\end{displaymath}
for any $\beta" \in \F_p$ . We have that
\begin{displaymath}
    \sum_{\beta \in B} \sum_{\gamma \in \Gamma^+} \bigg( \frac{\beta- \gamma}p \bigg) \leq \sum_{\beta = \beta'}\bigg( \frac{p-1}4 + t \bigg) + \sum_{\beta \neq \beta'} \bigg( \frac{p-1}4 - t +1 \bigg) =
\end{displaymath}
\begin{displaymath}
    \frac{(p-1)^2}{8} + \# \{ \beta \neq \beta' \}  + (\# \{ \beta = \beta' \}  - \#\{ \beta \neq \beta' \})t.
    \end{displaymath}
Since an element with the highest multiplicity appears at least once in the multiset:
\begin{displaymath}
    \# \{ \beta \neq \beta' \} \leq \frac{p-3}{2}
\end{displaymath}

which means by equation \eqref{eq:p^28}
\begin{displaymath}
    \# \{ \beta = \beta'\} - \#\{ \beta \neq \beta'\} \geq \sum_{\beta \in B} \sum_{\gamma \in \Gamma^+} \bigg( \frac{\beta- \gamma}p \bigg) - \frac{(p-1)^2}8 - \frac{p-3}{2} >   .
\end{displaymath}
\begin{displaymath}
  >  \frac{3p^2}{16} + \frac{\sqrt{p(p^2-1)}}4 - \frac{(p-1)^2}8 - \frac{p-3}{2}.
\end{displaymath}
If $p \geq 23$, then expression above is greater than 0, which means
\begin{displaymath}
     \# \{ \beta = \beta'\} - \#\{ \beta \neq \beta'\} > 0
\end{displaymath}
Thus there is a $\beta' \in \F_p$ , which has multiplicity greater than $\frac{p-1}{4}$ in $B$. Suppose that there is a $\gamma \in \F_p$, such that $(\frac{\beta'-\gamma}p) = -1$ and
\begin{displaymath}
    \sum_{\alpha \in A} \bigg(\frac{\alpha-\gamma}p \bigg) = \sum_{\beta \in B} \bigg(\frac{\beta-\gamma}p \bigg) + c-p.
\end{displaymath} 
It follows that:
\begin{displaymath}
      \sum_{\alpha \in A} \bigg(\frac{\alpha-\gamma}p \bigg) -  \sum_{\beta \neq \beta'} \bigg(\frac{\beta-\gamma}p \bigg)= - \# \{\beta = \beta'\} + c-p.
\end{displaymath}
Here $\sum_{\alpha \in A} \left(\frac{\alpha-\gamma}p \right) -  \sum_{\beta \neq \beta'} \left(\frac{\beta-\gamma}p \right) > -\frac{3(p-1)}4$ and $- \# \{\beta = \beta'\} + c-p < -\frac{3(p-1)}4$, which is a contradiction. This means that if for a $ \gamma$ value it holds that $\left(\frac{\beta'-\gamma}p \right) = -1 $, then
\begin{displaymath}
    \sum_{\alpha \in A } \bigg(\frac{\alpha-\gamma}p \bigg) = \sum_{\beta \in B} \bigg(\frac{\beta-\gamma}p \bigg) + c.
\end{displaymath}
Let $\Gamma^- \subset \mathbb{F}_p$ the set of $\gamma$ values, such that $\left(\frac{\beta'-\gamma}p \right) = -1$. Notice that $|\Gamma'|=\dd$. Therefore
\begin{displaymath}
   \sum_{\gamma \in \Gamma^-} \sum_{\alpha \in A} \bigg(\frac{\alpha-\gamma}p \bigg) = \sum_{\gamma \in \Gamma^-}\sum_{\beta \in B} \bigg(\frac{\beta-\gamma}p \bigg) + \frac{(p-1)c}2.
\end{displaymath}
Again, we rearrange the equation
\begin{equation}\label{eq:valami}
   \sum_{\gamma \in \Gamma^-} \Bigg( \sum_{\alpha \in A} \bigg(\frac{\alpha-\gamma}p \bigg)- \sum_{\beta \neq \beta'} \bigg(\frac{\beta-\gamma}p \bigg)  \Bigg) = -\frac{p-1}2 \# \{\beta = \beta'\} + \frac{(p-1)c}2.
\end{equation}
By Table 1.
\begin{equation}\label{eq:legfeljebbegy}
    \bigg| \sum_{\gamma \in \Gamma^-} \bigg( \frac{\alpha-\gamma}{p}\bigg) \bigg| \leq 1,
\end{equation}
for any $\alpha \neq \beta' $ value. If we change the sums on the left hand side of equation \ref{eq:valami}, then each summand is of absolute value at most $1$ by \eqref{eq:legfeljebbegy}. Thus
\begin{displaymath}
    \frac{p-1}2 \left| \#\{\beta = \beta' \} - c \right| \leq \frac{p-1}{2} + \#\{ \beta \neq \beta'\}.
\end{displaymath}
Since $ \#\{ \beta \neq \beta'\} < \frac{p-1}{4}$, we have that:
\begin{displaymath}
     | \#\{\beta = \beta' \} - c | < \frac{3}{2},
\end{displaymath}
and since $ | \#\{\beta = \beta' \} - c |$ is an integer, we have
\begin{displaymath}
     | \#\{\beta = \beta' \} - c | \leq 1.
\end{displaymath}
Let $\Gamma_+ \subset \mathbb F_p$ the set of elements of $\F_p$ such that   $(\frac{\beta'-\gamma}p) = 1$, and
\begin{displaymath}
    \sum_{\alpha \in A} \bigg(\frac{\alpha-\gamma}p \bigg) = \sum_{\beta \in B} \bigg(\frac{\beta-\gamma}p \bigg) + c.
\end{displaymath}
Now since $\#\{ \beta \neq \beta'\} < \#\{ \beta = \beta'\} $, for the elements of $\Gamma_+$ we have
\begin{displaymath}
    \sum_{\beta \in B} \bigg(\frac{\beta-\gamma}p \bigg) > 0, 
\end{displaymath}
so 
\begin{displaymath}
    \sum_{\beta \in B} \bigg(\frac{\beta-\gamma}p \bigg) + c > \frac{p-1}{4}.
\end{displaymath}
There is at most one $\gamma$ value, such that
\begin{displaymath}
    \bigg| \sum_{\alpha \in A} \bigg( \frac{\alpha-\gamma}{p}\bigg) \bigg| > \frac{p-1}{4}.
\end{displaymath}
By combining these two observations we obtain that $|\Gamma_+| \leq 1 $. 

We have seen that if $\left(\frac{b'-\gamma}{p} \right)=-1$, then $r=c$.
Thus $ |\Gamma_+| = \frac{p-1}2 -c $, so $c \geq \frac{p-3}2$. Since $|\#\{ \beta = \beta'\} - c | \leq 1 $, we have that $\# \{\beta = \beta'\} \geq \frac{p-5}2$. If $\gamma = \beta'$ then
\begin{displaymath}
    \bigg| \sum_{\beta \in B} \bigg(\frac{\beta-\beta'}p \bigg) \bigg| \leq 2.
\end{displaymath}
so
\begin{displaymath}
     \bigg| \sum_{\alpha \in A} \bigg( \frac{\alpha-\beta'}p\bigg) \bigg| \ge \frac{p-9}{2} > \frac{p-1}4.
\end{displaymath}
We have seen that there is at most one $\gamma$ value with this property and $\beta' \not\in \Gamma_+$ so 
 we have $|\Gamma_+| = 0$. Using again $|\Gamma_+|=\dd-c$ we obtain $c = \frac{p-1}2$.
\end{proof}

\medskip

\subsection{Unicity}
In this section we prove that  the leading coefficient of the polynomial of range sum $p$ of degree $\frac{p-1}{2}$ determines the polynomial itself. 

If $c = 1 $ then there is one $\gamma' \in \mathbb{F}_p$, where
\begin{displaymath}
    \sum_{\alpha \in A}\bigg( \frac{\alpha-\gamma'}p \bigg) = \sum_{\beta \in B}\bigg( \frac{\beta-\gamma'}p \bigg) + 1 - p.
\end{displaymath}
This means that $ \sum_{\alpha \in A}\left( \frac{\alpha-\gamma'}p \right) = -\sum_{\beta \in B}\left( \frac{\beta-\gamma'}p \right) = -\frac{p-1}2$, so $A = \{\alpha \in \mathbb{F}_p \colon \left( \frac{\alpha - \gamma}{p}\right) = -1 \}$. Thus
\begin{displaymath}
    f(x) = \prod_{\left(\frac{\alpha - \gamma}{p} \right) = -1} (x - \alpha)=(x+ \gamma)^{\frac{p-1}{2}}+1.
\end{displaymath}

If $c = \frac{p-1}{2}$ then we have seen that $|\#\{ \beta = \beta'\} - c | \leq 1 $, so $\#\{ \beta = \beta'\} \geq \frac{p-3}2$.
Thus for $\gamma = \beta'$, 
\begin{displaymath}
    \sum_{\alpha \in A} \bigg( \frac{\beta' - \gamma}p\bigg) \leq -\frac{p-5}2.
\end{displaymath}
Therefore  there are at least $\frac{p-3}2$ elements $\alpha$ in $A$ for which $\alpha- \beta'$ is a quadratic nonresidue. On the other hand in the $B$ multiset the  element $\beta'$ has multiplicity at least $\frac{p-3}2$. This means that there are at least $\frac{p-5}2$ elements  of $B$ for which $ \left( \frac{\beta-\beta'}p \right) = 1 $ and $f(x) = 1 $. Let us define the following polynomial:
\begin{displaymath}
    g(x) = \frac{p-1}2 (x^{\frac{p-1}2} + 1)
\end{displaymath}
The degree of the polynomial $f(x)- g(x)$   is at most $\frac{p-1}2$ but it has at least $\frac{p-5}{2} + \frac{p-5}{2}$ roots which means $f(x)= g(x) $.

\end{document}